\documentclass[a4paper]{amsart}
\usepackage{amsrefs}
\usepackage{fullpage}
\usepackage{amssymb,mathrsfs}
\usepackage[all]{xy}
\newdir{ >}{{}*!/-5pt/@{>}} 
\SelectTips{cm}{} 

\title{Boardman's whole-plane obstruction group \\
	for Cartan--Eilenberg systems}
\author{Gard Olav Helle and John Rognes}
\address{Department of Mathematics, University of Oslo, Norway}
\email{gardoh@student.matnat.uio.no}
\email{rognes@math.uio.no} \urladdr{http://folk.uio.no/rognes}
\subjclass[2010]{18G40}

\date{February 2nd 2018}

\newtheorem{theorem}{Theorem}[section]
\newtheorem{proposition}[theorem]{Proposition}
\newtheorem{lemma}[theorem]{Lemma}

\theoremstyle{definition}
\newtheorem{definition}[theorem]{Definition}

\theoremstyle{remark}
\newtheorem{example}[theorem]{Example}
\newtheorem{remark}[theorem]{Remark}

\numberwithin{equation}{section}
\numberwithin{figure}{section}

\let\ker\relax
\DeclareMathOperator{\ker}{Ker}
\DeclareMathOperator{\im}{Im}
\DeclareMathOperator{\cok}{Cok}
\DeclareMathOperator*{\colim}{colim}
\DeclareMathOperator{\cone}{cone}

\DeclareMathOperator*{\hocolim}{hocolim}
\DeclareMathOperator*{\holim}{holim}
\DeclareMathOperator{\id}{id}
\DeclareMathOperator{\Gap}{Gap}
\DeclareMathOperator*{\Rlim}{Rlim}

\newcommand{\bZ}{\mathbb{Z}}
\newcommand{\longto}{\longrightarrow}
\newcommand{\onto}{\twoheadrightarrow}
\newcommand{\cA}{\mathcal{A}}
\newcommand{\cC}{\mathcal{C}}
\newcommand{\cI}{\mathcal{I}}
\renewcommand{\:}{\colon}

\begin{document}
\begin{abstract}
Each extended Cartan--Eilenberg system $(H, \partial)$ gives rise to two
exact couples and one spectral sequence.  We show that the canonical
colim-lim interchange morphism associated to $H$ is a surjection,
and that its kernel is isomorphic to Boardman's whole-plane obstruction
group $W$, for each of the two exact couples.
\end{abstract}

\maketitle

\section{Introduction}

In a classic study, Boardman \cite{Boa99} analyzed the convergence of
the spectral sequence $E^r \Longrightarrow G$ arising from an exact
couple $(A, E)$ in the sense of Massey~\cite{Mas52}.  He identified a
condition on the exact couple, called conditional convergence, which
in the case of half-plane spectral sequences with exiting differentials
is sufficient to ensure strong convergence.  In the case of half-plane
spectral sequences with entering differentials, conditional convergence
and the vanishing of an obstruction group $RE^\infty$ guarantee strong
convergence.  Finally, in the case of whole-plane spectral sequences,
Boardman identified another obstruction group
$$
W = \colim_s \Rlim_r K_\infty \im^r A_s \,,
$$
such that conditional convergence and the vanishing of both $RE^\infty$
and $W$ imply strong convergence.

It is an insight of the first author that for spectral sequences
arising from Cartan--Eilenberg systems $(H, \partial)$, as defined in
\cite{CE56}, there is a natural isomorphism
$$
W \cong \ker(\kappa) \,,
$$
where $\kappa$ is the canonical colim-lim interchange morphism
$$
\kappa \: \colim_i \lim_j H(i,j) \longto \lim_j \colim_i H(i,j) \,.
$$
We prove this in Theorems~\ref{thm:W''iskerkappa} and
\ref{thm:W'iskerkappa}.  Furthermore, $\kappa$ is always a surjection,
cf.~Theorem~\ref{thm:lambda-kappa-sequence}.  We find that this
direct relationship to the interchange morphism clarifies the role of
Boardman's whole-plane obstruction group~$W$.  For example, we obtain
a straightforward proof of his criterion for the vanishing of $W$,
see Proposition~\ref{prop:kerkappazero}.

In order to introduce notations, we review spectral sequences,
exact couples and Cartan--Eilenberg systems in Sections~2, 3 and~4,
respectively.  Our treatment of convergence is very close to that
of \cite{Boa99}, but we obtain slightly more general conclusions in
particular cases.

Our novel work begins in Section~5, where we introduce the canonical
interchange morphism $\kappa$ and show that it is always surjective for
Cartan--Eilenberg systems.  In Sections~6 and~7 we consider right and
left Cartan--Eilenberg systems, and their associated exact couples $(A'',
E^1)$ and $(A', E^1)$, and identify the kernel of $\kappa$ with the
respective whole-plane obstruction groups.  In Section 8 we give some
examples of spectral sequences that arise from biinfinite sequences of
spectra, illustrating that the whole-plane obstruction $W$ can be highly
nontrivial, and giving topological interpretations of its meaning.


\section{Spectral sequences}

Let $R$ be a ring and let $\cA$ be the graded abelian category of graded
$R$-modules $M = (M_t)_t$, where $t \in \bZ$ is the \emph{internal
degree}.  (Greater abstraction is possible, but beware the counterexample
of Neeman--Deligne \cite{Nee02} to Roos' theorem \cite{Roo61}, as repaired
in \cite{Roo06}.)

\begin{definition}[Leray \cite{Ler46}, Koszul \cite{Kos47}]
A \emph{spectral sequence} is a sequence $(E^r, d^r)_r$ of differential
graded objects in $\cA$, for $r\ge1$, equipped with isomorphisms
$E^{r+1} \cong H(E^r, d^r)$.  Each term $E^r = (E^r_s)_s$ is a graded
object in $\cA$, where $s \in \bZ$ is the \emph{filtration degree}.
We assume that each differential $d^r = (d^r_s)_s$, with $d^r_s \: E^r_s
\to E^r_{s-r}$, has filtration degree~$-r$ and internal degree~$-1$.
It satisfies $d^r \circ d^r = 0$, and its homology $H(E^r, d^r)$ is
the graded object with $H_s(E^r, d^r) = \ker(d^r_s) / \im(d^r_{s+r})$
in $\cA$.
\end{definition}

\begin{lemma} \label{lem:ascdesc}
For any spectral sequence $(E^r, d^r)_r$ there are subobjects
$$
0 = B^1_s \subset \dots \subset B^r_s \subset \dots
\subset Z^r_s \subset \dots \subset Z^1_s = E^1_s
$$
in $\cA$, with $E^r_s \cong Z^r_s/B^r_s$ for all integers $r\ge1$ and $s$.
\end{lemma}

\begin{proof}
This is clear for $r=1$.  By induction on $r$ we may assume that
$E^r_s \cong Z^r_s/B^r_s$ for all $s \in \bZ$.  Then $\im(d^r_s)
\subset \ker(d^r_s) \subset E^r_s$ correspond to $B^{r+1}_s/B^r_s
\subset Z^{r+1}_s/B^r_s \subset Z^r_s/B^r_s$ for well-defined subobjects
$B^{r+1}_s$ and $Z^{r+1}_s$ of $Z^r_s$, with $B^r_s \subset B^{r+1}_s
\subset Z^{r+1}_s \subset Z^r_s$.  This completes the inductive step.
\end{proof}

\begin{definition}
For each integer~$s$, let
\begin{align*}
Z^\infty_s &= \lim_r Z^r_s \\
B^\infty_s &= \colim_r B^r_s \\
E^\infty_s &= Z^\infty_s / B^\infty_s
\end{align*}
denote the infinite cycles, the infinite boundaries, and the
$E^\infty$-term, respectively.  Following Boardman \cite{Boa99}*{(5.1)}
let
$$
RE^\infty_s = \Rlim_r Z^r_s \,.
$$
\end{definition}

\begin{remark}
Fix a filtration degree~$s$.  If $d^r_s = 0$ for all sufficiently large
$r$, then the descending sequence $(Z^r_s)_r$ is eventually constant and
$RE^\infty_s = 0$.  This argument may be applied in one internal degree
$t$ at a time.
\end{remark}

\begin{definition}
A \emph{filtration} of an object $G$ in $\cA$ is a diagram of
subobjects
$$
\dots \subset F_{s-1} G \subset F_s G \subset \dots \subset G \,,
$$
where $s \in \bZ$.  Let $F_\infty G = \colim_s F_s G$, $F_{-\infty} G =
\lim_s F_s G$ and $RF_{-\infty} G = \Rlim_s F_s G$.  The filtration
is \emph{exhaustive} if $F_\infty G \to G$ is an isomorphism, it
is \emph{Hausdorff} if $F_{-\infty} G = 0$, and \emph{complete} if
$RF_{-\infty} G = 0$.
\end{definition}

\begin{lemma}
For an exhaustive complete Hausdorff filtration we can recover
$G$ from the filtration subquotients $F_j G/F_i G$
for $i \le j$ in $\bZ$, by either one of the isomorphisms
$$
\colim_j \lim_i F_j G/F_i G
	\ \cong \ G
	\ \cong \ \lim_i \colim_j F_j G/F_i G \,.
$$
\end{lemma}

\begin{proof}
This follows from the isomorphisms $F_j G \cong \lim_i F_j G/F_i G$
and $G / F_i G \cong \colim_j  F_j G/F_i G$, respectively.
\end{proof}

\begin{definition}[\cite{CE56}*{XV.2}, \cite{Boa99}*{5.2}]
A spectral sequence $(E^r, d^r)_r$ \emph{converges weakly} to a filtered
object~$G$ if the filtration is exhaustive and there are isomorphisms
$$
E^\infty_s \cong F_s G / F_{s-1} G
$$
for each integer~$s$.  The spectral sequence \emph{converges} if it
converges weakly and the filtration is Hausdorff.  It \emph{converges
strongly} if it converges and the filtration is complete.
\end{definition}

\begin{remark}
For a strongly convergent spectral sequence we can recover the target
$G$ from the $E^\infty$-term if we are able to resolve the extension
problems, i.e., to determine the diagram of filtration subquotients
$F_j G/F_i G$ from knowledge of the minimal subquotients $F_s G/F_{s-1}
G \cong E^\infty_s$ for $i < s \le j$ and other available information.
\end{remark}

\section{Exact couples}

\begin{definition}[\cite{Mas52}*{\S4}]
An \emph{exact couple} $(A, E, \alpha, \beta, \gamma)$ in $\cA$ is a pair
of graded objects $A = (A_s)_s$ and $E = (E_s)_s$, and three morphisms
$\alpha = (\alpha_s \: A_{s-1} \to A_s)_s$, $\beta = (\beta_s \: A_s
\to E_s)_s$, $\gamma = (\gamma_s \: E_s \to A_{s-1})_s$ of filtration
degree $+1$, $0$ and $-1$, respectively, such that the triangle
$$
\xymatrix{
A_{s-1} \ar[rr]^-{\alpha_s} && A_s \ar[dl]^-{\beta_s} \\
& E_s \ar[ul]^-{\gamma_s}
}
$$
is exact for each $s \in \bZ$.
We assume that the internal degree of $\alpha$ is $0$, and that the
internal degrees of $\beta$ and $\gamma$ are either $0$ and $-1$,
or $-1$ and $0$.
\end{definition}

\begin{remark}
Equivalently we require that each triangle in the diagram
$$
\xymatrix{
\dots \ar[r]
& A_{s-2} \ar[rr]^-{\alpha_{s-1}}
	&& A_{s-1} \ar[rr]^-{\alpha_s} \ar[dl]^-{\beta_{s-1}}
	&& A_s \ar[rr]^-{\alpha_{s+1}} \ar[dl]^-{\beta_s}
	&& A_{s+1} \ar[r] \ar[dl]^-{\beta_{s+1}}
	& \dots \\
&& E_{s-1} \ar[ul]^-{\gamma_{s-1}}
&& E_s \ar[ul]^-{\gamma_s}
&& E_{s+1} \ar[ul]^-{\gamma_{s+1}}
}
$$
is exact.  This object is called an \emph{unrolled}, or unraveled,
exact couple \cite{Boa99}*{\S0}.
\end{remark}

\begin{definition} \label{def:excoupleZrBrEr}
Let $(A, E, \alpha, \beta, \gamma)$ be an exact couple.
For integers $r\ge1$ and $s$ let
\begin{align*}
Z^r_s &= \gamma^{-1} \im(\alpha^{r-1} \: A_{s-r} \to A_{s-1}) \\
B^r_s &= \beta \ker(\alpha^{r-1} \: A_s \to A_{s+r-1}) \\
E^r_s &= Z^r_s / B^r_s \,.
\end{align*}
Let $d^r_s \: E^r_s \to E^r_{s-r}$ be given by $d^r_s([x]) = [\beta(y)]$
where $\gamma(x) = \alpha^{r-1}(y)$.
\end{definition}

\begin{lemma}
$\ker(d^r_s) = Z^{r+1}_s/B^r_s$ and $\im(d^r_{s+r}) = B^{r+1}_s/B^r_s$,
so $H_s(E^r, d^r) \cong E^{r+1}_s$ and
$(E^r, d^r)_r$ is a spectral sequence.
\qed
\end{lemma}

\begin{remark}
The objects $Z^r_s$ and $B^r_s$ of Definition~\ref{def:excoupleZrBrEr} agree
with those associated in~Lemma~\ref{lem:ascdesc} to the spectral
sequence $(E^r, d^r)_r$.
\end{remark}

\begin{definition}
Given a sequence $\dots \to A_{s-1} \overset{\alpha_s}\longto A_s
\to \dots$, consider the colimit $A_\infty = \colim_s A_s$, the limit
$A_{-\infty} = \lim_s A_s$, and the derived limit $RA_{-\infty} = \Rlim_s
A_s$.  Let $\iota_s \: A_s \to A_\infty$ and $\pi_s \: A_{-\infty} \to
A_s$ be the colimit and limit structure maps, respectively.  The colimit
$A_\infty$ is filtered, for $s \in \bZ$, by
$$
F_s A_\infty = \im(\iota_s \: A_s \to A_\infty) \,.
$$
The limit $A_{-\infty}$ is filtered, also for $s \in \bZ$, by
$$
F_s A_{-\infty} = \ker(\pi_s \: A_{-\infty} \to A_s) \,.
$$
\end{definition}

\begin{definition}[\cite{Boa99}*{5.10}]
An exact couple $(A, E, \alpha, \beta, \gamma)$ \emph{converges
conditionally to the colimit} if $A_{-\infty} = 0$ and $RA_{-\infty}
= 0$.  It \emph{converges conditionally to the limit} if $A_\infty = 0$.
\end{definition}

\begin{remark}
One often says that a spectral sequence is conditionally convergent, but
conditional convergence is, strictly speaking, a property of an exact
couple.
\end{remark}

The following two variants of Boardman's results show that conditional
convergence, when combined with the vanishing of $RE^\infty$, suffices
to give quite good convergence results.  In the first case the target is
correct, but we may not get strong convergence because the filtration
(with limit $F_{-\infty} A_\infty$) might not be Hausdorff.  In the
second case we get strong convergence, but the target is sometimes only
a quotient (by $RA_{-\infty}$) of the most desirable target object.
In both cases the error term is given by the whole-plane obstruction
object $W$ for the exact couple, whose definition is reviewed directly
after the two theorems.

\begin{theorem}[cf.~\cite{Boa99}*{8.10}] \label{thm:GH2.2.5}
Let $(A, E, \alpha, \beta, \gamma)$ be an exact couple that converges
conditionally to the colimit, and assume that $RE^\infty
= 0$.  Then the associated spectral sequence converges weakly to
$A_\infty$, and the filtration of $A_\infty$ is complete.
Moreover, $F_{-\infty} A_\infty \cong W$.
\end{theorem}

\begin{theorem}[cf.~\cite{Boa99}*{8.13}] \label{thm:GH2.2.4}
Let $(A, E, \alpha, \beta, \gamma)$ be an exact couple that converges
conditionally to the limit, and assume that $RE^\infty
= 0$.  Then the associated spectral sequence converges strongly to
$A_{-\infty}$.  Moreover, $RA_{-\infty} \cong W$.
\end{theorem}

\begin{definition}[\cite{Boa99}*{\S3, \S8}]
\label{def:Q-RQ-W}
Let $(A, E, \alpha, \beta, \gamma)$ be any exact couple.
For integers $r\ge1$ and $s$ let
\begin{align*}
\im^r A_s &= \im(\alpha^r \: A_{s-r} \to A_s) \\
Q_s &= \lim_r \im^r A_s \\
RQ_s &= \Rlim_r \im^r A_s \\
K_\infty \im^r A_s &= \ker(\iota_s \: A_s \to A_\infty) \cap \im^r A_s \\
W &= \colim_s \Rlim_r K_\infty \im^r A_s \,.
\end{align*}
\end{definition}

\begin{remark}
Boardman employs transfinite induction to define image subsequences
$\im^\sigma A_s$ for arbitrary ordinals $\sigma$, and uses these to
prove that the sufficient conditions he gives for strong convergence are
also necessary.  We will only establish the sufficiency of our conditions,
and for these results there is no need to invoke transfinite induction.
\end{remark}

\begin{lemma}[cf.~\cite{Boa99}*{5.4(b)}]
\label{lem:Boa5.4b}
The filtration of
$A_{-\infty} = \lim_s \im(\pi_s) = \lim_s Q_s$ is complete
and Hausdorff.
\end{lemma}

\begin{proof}
The colimit structure map $\pi_s \: A_{-\infty} \to A_s$ factors through
$\im(\pi_s) \subset \im^r A_s \subset A_s$ for each $r$, hence also
through $\im(\pi_s) \subset Q_s \subset A_s$.  Therefore the identity
map of $A_{-\infty}$ factors through $\lim_s \im(\pi_s) \subset \lim_s
Q_s \subset \lim_s A_s$, which implies that $\lim_s
\im(\pi_s) = \lim_s Q_s = \lim_s A_s$.  The short exact sequences
$0 \to F_s A_{-\infty} \to A_{-\infty} \to \im(\pi_s) \to 0$ give
an exact sequence
$$
0 \to F_{-\infty} A_{-\infty} \longto A_{-\infty}
\overset{\cong}\longto \lim_s \im(\pi_s)
\longto RF_{-\infty} A_{-\infty} \longto 0
$$
upon passage to limits.
Hence $F_{-\infty} A_{-\infty} = 0$ and $RF_{-\infty} A_{-\infty} = 0$.
\end{proof}

\begin{lemma} \label{lem:ACB}
Consider any two morphisms
$\xymatrix{
A & C \ar[l]_-f \ar[r]^-g & B
}$
in $\cA$.  There is an isomorphism
$$
\im(f)/f(\ker(g))
	\overset{\cong}\longto \im(g)/(g(\ker(f))
$$
given by $[a] \mapsto [g(c)]$, where $f(c) = a$.
\end{lemma}

\begin{proof}
The isomorphism factors through $C/(\ker(f)+\ker(g))$.
\end{proof}

\begin{lemma}[cf.~\cite{Boa99}*{5.6}]
\label{lem:Boa5.6}
(a) There is a natural short exact sequence
$$
0 \to F_s A_\infty/F_{s-1} A_\infty
	\longto E^\infty_s \longto Z^\infty_s/\ker(\gamma) \to 0 \,.
$$
(b) There is a natural six term exact sequence
$$
0 \to Z^\infty_s/\ker(\gamma)
        \overset{\gamma}\longto Q_{s-1}
        \overset{\alpha}\longto Q_s
        \to RE^\infty_s
        \overset{\gamma}\longto RQ_{s-1}
        \overset{\alpha}\longto RQ_s \to 0 \,.
$$
(c)
If $RE^\infty = 0$ then $\lim_s Q_s \to Q_s$ is surjective.
\end{lemma}

\begin{proof}
(a)
The inclusions $B^\infty_s \subset \im(\beta) = \ker(\gamma) \subset
Z^\infty_s$ lead to a short exact sequence
$$
0 \to \im(\beta)/B^\infty_s \to E^\infty_s
	\to Z^\infty_s/\ker(\gamma) \to 0 \,.
$$
By Lemma~\ref{lem:ACB} applied to the two morphisms
$\xymatrix{
E^1_s & A_s \ar[l]_-{\beta} \ar[r]^-{\iota_s} & A_\infty
}$
there is an isomorphism
$$
\im(\beta)/B^\infty_s \cong F_s A_\infty/F_{s-1} A_\infty \,.
$$

(b)
The short exact sequences
$$
0 \to Z^r_s/\ker(\gamma) \overset{\gamma}\longto
	\im^{r-1} A_{s-1} \overset{\alpha}\longto \im^r A_s \to 0
$$
give the stated six term exact sequence upon passage to limits.

(c)
If $RE^\infty_s = 0$ for all $s$, then $\alpha \: Q_{s-1} \to Q_s$
is surjective for each $s$.  This implies that $\lim_s Q_s \to Q_s$
is surjective for each $s$.
\end{proof}

\begin{proof}[Proof of Theorem~\ref{thm:GH2.2.5}]
We are assuming that $A_{-\infty} = 0$, $RA_{-\infty} = 0$ and
$RE^\infty = 0$.  The filtration of $A_\infty$ is exhaustive
because
$$
F_\infty A_\infty = \colim_s \im(A_s \to A_\infty)
\cong \im(\colim_s A_s \to A_\infty) = A_\infty \,.
$$
It is also complete, because the derived limit of the surjections
$A_s \onto F_s A_\infty$ is a surjection $0 = RA_{-\infty} \onto
RF_{-\infty} A_\infty$.  By Lemma~\ref{lem:Boa5.4b}, $\lim_s Q_s =
A_{-\infty} = 0$.  Lemma~\ref{lem:Boa5.6} then implies that $Q_s = 0$
for all $s \in \bZ$.  Furthermore, $RQ_{s-1} \cong RQ_s$,
$Z^\infty_s/\ker(\gamma) = 0$ and $F_s A_\infty/F_{s-1} A_\infty \cong
E^\infty_s$, for all $s \in \bZ$.

The short exact sequences
$$
0 \to K_\infty \im^r A_s \to \im^r A_s
	\overset{\iota_s|}\longto F_{s-r} A_\infty \to 0
$$
give the exact sequence
$$
Q_s \to F_{-\infty} A_\infty
	\to \Rlim_r K_\infty \im^r A_s \to RQ_s
$$
upon passage to limits over~$r$.
The Mittag--Leffler short exact sequence
\begin{equation} \label{eq:MLses}
0 \to \Rlim_s Q_s \longto RA_{-\infty} \longto \lim_s RQ_s \to 0
\end{equation}
of \cite{Boa99}*{3.4(b)}
simplifies to $0 = RA_{-\infty} \cong \lim_s
RQ_s \cong RQ_s$.  Hence $F_{-\infty} A_\infty \cong \Rlim_r K_\infty
\im^r A_s$ for all~$s$, and $F_{-\infty} A_\infty \cong W$.
\end{proof}

\begin{proof}[Proof of Theorem~\ref{thm:GH2.2.4}]
We are assuming that $A_\infty = 0$ and $RE^\infty = 0$.  The filtration
of $A_{-\infty}$ is complete and Hausdorff by Lemma~\ref{lem:Boa5.4b}.
It is exhaustive because
$$
F_\infty A_{-\infty}
= \colim_s \ker(A_{-\infty} \to A_s)
\cong \ker(A_{-\infty} \to \colim_s A_s)
= A_{-\infty} \,,
$$
since $\colim_s A_s = A_\infty = 0$.

By Lemma~\ref{lem:Boa5.6} we have short exact sequences
$$
0 \to E^\infty_s \overset{\gamma}\longto Q_{s-1}
	\overset{\alpha}\longto Q_s \to 0
$$
and isomorphisms $\alpha \: RQ_{s-1} \cong RQ_s$, for all $s \in
\bZ$.  Furthermore, $\lim_s Q_s \to Q_s$ is surjective.  Hence,
by Lemma~\ref{lem:Boa5.4b} the image of $\pi_s \: A_{-\infty} \to
A_s$ is equal to $Q_s$.  The surjections $A_{-\infty} \onto Q_{s-1}
\overset{\alpha}\onto Q_s$ lead to the short exact sequence
$$
0 \to F_{s-1} A_{-\infty} \longto
F_s A_{-\infty} \longto \ker(\alpha) \to 0 \,.
$$
Thus $F_s A_{-\infty} / F_{s-1} A_{-\infty} \cong \ker(\alpha)
\cong E^\infty_s$.

The Mittag--Leffler sequence~\eqref{eq:MLses}
simplifies to an isomorphism $RA_{-\infty} \cong
\lim_s RQ_s \cong \colim_s RQ_s$.  Furthermore, $K_\infty \im^r A_s
= \im^r A_s$, so $\Rlim_r K_\infty \im^r A_s = RQ_s$
and $W = \colim_s RQ_s$.
\end{proof}

\section{Cartan--Eilenberg systems}

\begin{definition}
Let $\cI$ be a linearly ordered set, and let $\cI^{[1]}$ be its
\emph{arrow category}, with one object $(i,j)$ for each pair in $\cI$
with $i \le j$, and a single morphism from $(i,j)$ to $(i',j')$, where
$i'\le j'$, precisely when $i \le i'$ and $j \le j'$.

Let $\cI^{[2]}$ be the category with one object $(i,j,k)$ for each
triple in $\cI$ with $i \le j \le k$, and a single morphism $(i,j,k)
\to (i',j',k')$, where $i' \le j' \le k'$, precisely when $i\le i'$,
$j\le j'$ and $k\le k'$.

Let $d_0$, $d_1$ and $d_2 \: \cI^{[2]} \to \cI^{[1]}$ be the functors
mapping $(i,j,k)$ to $(j,k)$, $(i,k)$ and $(i,j)$, respectively.
There are natural transformations $\iota \: d_2 \to d_1$ and $\pi \:
d_1 \to d_0$, with components $\iota \: (i,j) \to (i,k)$ and $\pi \:
(i,k) \to (j,k)$, respectively.
\end{definition}

View the set $\bZ$ of integers as linearly ordered, with the usual
ordering.

\begin{definition}[\cite{CE56}*{XV.7}]
\label{def:I-system}
An \emph{$\cI$-system} $(H, \partial)$ in $\cA$ is a functor $H \:
\cI^{[1]} \to \cA$ and a natural transformation $\partial \: H d_0
\to H d_2$ of functors $\cI^{[2]} \to \cA$,
such that the triangle
$$
\xymatrix{
H d_2 \ar[r]^-{H\iota} & H d_1 \ar[d]^-{H\pi} \\
& H d_0 \ar[ul]^-{\partial} 
}
$$
is exact.  We assume that the internal degrees of $H\iota$ and $H\pi$
are $0$, and that $\partial$ has internal degree~$-1$.  A $\bZ$-system is
called a (homological) \emph{Cartan--Eilenberg system}.
\end{definition}

\begin{remark}
The functor $H$ assigns an object $H(i,j)$ to each pair $(i,j)$ with $i
\le j$ in $\cI$, and a morphism
$$
\eta \: H(i,j) \longto H(i',j')
$$
to each morphism $(i,j) \to (i',j')$.  By functoriality, $\eta \: H(i,j)
\to H(i,j)$ is the identity, and the composite $\eta \circ \eta \: H(i,j)
\to H(i',j') \to H(i'',j'')$ is equal to $\eta \: H(i,j) \to H(i'',j'')$.
In particular, $\eta = H\pi \circ H\iota \: H(i,j) \to
H(i',j')$ for $\iota \: (i,j) \to (i,j')$ and $\pi \: (i,j') \to (i',j')$.
The natural transformation $\partial$ has components
$$
\partial_{(i,j,k)} \: H(j,k) \longto H(i,j)
$$
for each triple $(i,j,k)$ with $i \le j \le k$ in $\cI$, and
$$
\eta \circ \partial_{(i,j,k)} = \partial_{(i',j',k')} \circ \eta
	\: H(j,k) \to H(i',j')
$$
whenever there is a morphism $(i,j,k) \to (i',j',k')$.
The triangle
\begin{equation} \label{eq:Hijk-exact-triangle}
\xymatrix{
H(i,j) \ar[r]^-{\eta} & H(i,k) \ar[d]^-{\eta} \\
& H(j,k) \ar[ul]^-{\partial}
}
\end{equation}
is exact, where $\partial = \partial_{(i,j,k)}$.
In particular, $H(i,i) = 0$ for each~$i$ in $\cI$.
\end{remark}

\begin{example}
In the language of \cite{LurHA}*{1.2.2}: If $\cC$ is a stable
$\infty$-category equipped with a t-structure, and if $X \in \Gap(\cI,
\cC)$ is an $\cI$-complex in $\cC$, then $H = \pi_*(X)$ with $H(i,j)
= (\pi_t X(i,j))_t$ is an $\cI$-system.  In particular, if $X$ is a
$\bZ$-complex then $H = \pi_*(X)$ is a Cartan--Eilenberg system.
\end{example}

\begin{definition} \label{def:ZrBrEr}
Let $(H, \partial)$ be a Cartan--Eilenberg system.
For each integer $s$ let $E^1_s = H(s-1,s)$, and
for each integer $r\ge1$ let
\begin{align*}
Z^r_s & = \ker(\partial \: E^1_s \to H(s-r,s-1)) \\
B^r_s &= \im(\partial \: H(s,s+r-1) \to E^1_s) \\
E^r_s &= Z^r_s / B^r_s
\end{align*}
define the $r$-cycles, $r$-boundaries and $E^r$-term, respectively.
These form sequences
$$
0 = B^1_s \subset \dots \subset B^r_s \subset \dots
\subset Z^r_s \subset \dots \subset Z^1_s = E^1_s \,,
$$
so each $E^r$-term is a subquotient of the $E^1$-term.
\end{definition}

\begin{remark}
We note that
\begin{align*}
Z^r_s &= \im(\eta \: H(s-r,s) \to E^1_s) \\
B^r_s &= \ker(\eta \: E^1_s \to H(s-1,s+r-1)) \,,
\end{align*}
by exactness.
Furthermore, $\partial \: E^1_s \to H(s-r,s-1)$ factors
through $\eta \: E^1_s \to H(s-1,s+r-1)$ by naturality, so $B^r_s
\subset Z^r_s$ by exactness.
$$
\xymatrix{
H(s-r,s-1) \ar[r]^-{\eta} & H(s-r,s) \ar[dd]^-{\eta} \\
\\
& E^1_s \ar[rr]^-{\eta} \ar[uul]^-{\partial}
	&& H(s-1,s+r-1) \ar[d]^-{\eta} \\
& && H(s,s+r-1) \ar[ull]^-{\partial}
}
$$
\end{remark}

\begin{lemma}
For integers $r\ge1$ and $s$ there is an isomorphism
$$
Z^r_s/Z^{r+1}_s \overset{\cong}\longto B^{r+1}_{s-r}/B^r_{s-r}
$$
given by $[x] \mapsto [\partial(\tilde x)]$ for $\eta(\tilde x) = x$,
where $\eta \: H(s-r,s) \to E^1_s$ and $\partial \: H(s-r,s) \to E^1_{s-r}$.
\end{lemma}

\begin{proof}
Apply Lemma~\ref{lem:ACB} to the two morphisms
$$
\xymatrix{
E^1_s & H(s-r,s) \ar[l]_-{\eta} \ar[r]^-{\partial} & E^1_{s-r}
\,,
}
$$
noting that $\im(\eta) = Z^r_s$,
$\eta(\ker(\partial)) = Z^{r+1}_s$,
$\im(\partial) = B^{r+1}_{s-r}$ and
$\partial(\ker(\eta)) = B^r_{s-r}$.
\end{proof}

\begin{definition}
For integers $r\ge1$ and $s$ let $d^r_s \: E^r_s \longto E^r_{s-r}$
be the composite morphism
$$
\xymatrix{
E^r_s = Z^r_s/B^r_s \ar@{->>}[r]
	& Z^r_s/Z^{r+1}_s \ar[r]^-{\cong}
	& B^{r+1}_{s-r}/B^r_{s-r} \ar@{ >->}[r]
	& Z^r_{s-r}/B^r_{s-r} = E^r_{s-r} \,.
}
$$
More explicitly, it is given by $[x] \mapsto [\partial(\tilde x)]$, where $x
\in Z^r_s \subset E^1_s$ and $\tilde x \in H(s-r,s)$ satisfy $\eta(\tilde
x) = x$, and $\partial(\tilde x) \in Z^r_{s-r} \subset E^1_{s-r}$.
\end{definition}

\begin{proposition}
For each Cartan--Eilenberg system $(H, \partial)$, the
associated sequence $(E^r, d^r)_{r\ge1}$ is a spectral sequence.
More precisely, $\ker(d^r_s) = Z^{r+1}_s/B^r_s$ contains $\im(d^r_{s+r})
= B^{r+1}_s/B^r_s$, so $d^r_s \circ d^r_{s+r} = 0$, and there is an
isomorphism
$$
E^{r+1}_s \cong H_s(E^r, d^r)
$$
given by $[x] \mapsto [\bar x]$, where $x \in Z^{r+1}_s$ maps
to $\bar x \in Z^{r+1}_s/B^r_s$.
\end{proposition}

\begin{proof}
The calculation of $\ker(d^r_s)$ and $\im(d^r_{s-r})$ is
evident from the definition of $d^r_s$.  The isomorphism in
question is the Noether isomorphism $Z^{r+1}_s/B^{r+1}_s \cong
(Z^{r+1}_s/B^r_s)/(B^{r+1}_s/B^r_s)$.
\end{proof}

\begin{remark}
The objects $Z^r_s$ and $B^r_s$ of Definition~\ref{def:ZrBrEr} agree
with those associated in~Lemma~\ref{lem:ascdesc} to the spectral
sequence $(E^r, d^r)_r$.
\end{remark}

\section{The canonical interchange morphism}

\begin{definition}
Given any functor $H \: \bZ^{[1]} \to \cA$, let $\kappa$ be the canonical
\emph{interchange morphism}
$$
\kappa \: \colim_j \lim_i H(i,j)
	\longto \lim_i \colim_j H(i,j) \,,
$$
as defined in \cite{ML98}*{IX.2(3)}.  The restriction $\kappa \iota_j$ of
$\kappa$ to $\lim_i H(i,j)$ is the limit over~$i$ of the colimit structure
maps $\iota_{i,j} \: H(i,j) \to \colim_j H(i,j)$.  The projection $\pi_i
\kappa$ of $\kappa$ to $\colim_j H(i,j)$ is the colimit over~$j$ of the
limit structure maps $\pi_{i,j} \: \lim_i H(i,j) \to H(i,j)$.
$$
\xymatrix{
\lim_i H(i, j) \ar[r]^-{\iota_j} \ar[dd]^-{\pi_{i,j}}
        & \colim_j \lim_i H(i, j) \ar[dr]^-{\kappa} \\
& & \lim_i \colim_j H(i,j) \ar[d]^-{\pi_i} \\
H(i,j) \ar[rr]^-{\iota_{i,j}}
        && \colim_j H(i,j) 
}
$$
\end{definition}

\begin{theorem} \label{thm:lambda-kappa-sequence}
Let $(H, \partial)$ be a Cartan--Eilenberg system.
There is a natural exact sequence
$$
0 \to \colim_j \Rlim_i H(i,j) \overset{\lambda}\longto
	\Rlim_i \colim_j H(i,j) \longto
	\colim_j \lim_i H(i,j) \overset{\kappa}\longto
	\lim_i \colim_j H(i,j) \to 0 \,,
$$
where $\cok(\lambda) = \colim_j \Rlim_i \cok(\iota_{i,j}) \cong \colim_j
\lim_i \ker(\iota_{i,j}) = \ker(\kappa)$ and the middle morphism has
internal degree~$-1$.  In particular, the interchange morphism $\kappa$
is surjective.
\end{theorem}

\begin{proof}
To simplify the notation, let
$$
D_i = \colim_k H(i,k)
$$
for each integer $i$.  The colimit over $k$
of~\eqref{eq:Hijk-exact-triangle} is the exact triangle
$$
\xymatrix{
H(i,j) \ar[r]^-{\iota_{i,j}} & D_i \ar[d]^-{\eta} \\
	& D_j \ar[ul]^-{\partial} \,.
}
$$
Let $K(i,j) = \ker(\iota_{i,j})$, $I(i,j) = \im(\iota_{i,j})$ and $C(i,j)
= \cok(\iota_{i,j})$.  We then have short exact sequences
\begin{align*}
0 &\to K(i,j) \longto H(i,j) \longto I(i,j) \to 0 \,, \\
0 &\to I(i,j) \longto D_i \longto C(i,j) \to 0 \,, \\
0 &\to C(i,j) \longto D_j \overset{\partial}\longto K(i,j) \to 0 \,,
\end{align*}
for each $i \le j$.  Passing to limits over $i$ we obtain the exact
sequences
\begin{align*}
0 &\to \lim_i K(i,j) \longto \lim_i H(i,j) \longto \lim_i I(i,j)
	\overset{\delta}\longto \Rlim_i K(i,j) \longto
	\Rlim_i H(i,j) \longto \Rlim_i I(i,j) \to 0 \,, \\
0 &\to \lim_i I(i,j) \longto \lim_i D_i \longto \lim_i C(i,j)
	\overset{\delta}\longto \Rlim_i I(i,j) \longto
	\Rlim_i D_i \longto \Rlim_i C(i,j) \to 0 \,, \\
0 &\to \lim_i C(i,j) \longto D_j \longto \lim_i K(i,j)
	\overset{\delta}\longto \Rlim_i C(i,j) \longto
	0 \longto \Rlim_i K(i,j) \to 0 \,.
\end{align*}
Here
$\lim_i D_j = D_j$ and $\Rlim_i D_j = 0$, so $\Rlim_i K(i,j) = 0$.
Passing to colimits over $j$ we get the exact sequences
\begin{align*}
0 &\to \colim_j \lim_i K(i,j) \longto \colim_j \lim_i H(i,j) \longto
	\colim_j \lim_i I(i,j) \to 0 \,, \\
0 &\to \colim_j \Rlim_i H(i,j) \overset{\cong}\longto
	\colim_j \Rlim_i I(i,j) \to 0 \,, \\
0 &\to \colim_j \lim_i I(i,j) \longto \lim_i D_i \longto
	\colim_j \lim_i C(i,j) \overset{\delta}\longto \\
&\qquad\qquad \overset{\delta}\longto \colim_j \Rlim_i I(i,j) \longto
	\Rlim_i D_i \longto \colim_j \Rlim_i C(i,j) \to 0 \,, \\
0 &\to \colim_j \lim_i C(i,j) \longto \colim_j D_j \longto
	\colim_j \lim_i K(i,j) \overset{\delta}\longto
	\colim_j \Rlim_i C(i,j) \to 0 \,.
\end{align*}
Here $\colim_j D_j = 0$, since $H(j,j) = 0$ for each $j$.  This gives
us a a natural isomorphism
$$
\delta \: \colim_j \lim_i K(i,j) \overset{\cong}\longto
        \colim_j \Rlim_i C(i,j)
$$
of internal degree~$+1$.
Furthermore, $\colim_j \lim_i C(i,j) = 0$, so $\colim_j \lim_i I(i,j)
\cong \lim_i D_i$.  We therefore have natural short exact sequences
\begin{align*}
0 &\to \colim_j \lim_i K(i,j) \longto \colim_j \lim_i H(i,j)
	\overset{\kappa}\longto \lim_i D_i \to 0 \,, \\
0 &\to \colim_j \Rlim_i H(i,j) \overset{\lambda}\longto
        \Rlim_i D_i \longto \colim_j \Rlim_i C(i,j) \to 0 \,.
\end{align*}
By using $\delta^{-1}$ to splice these together, we obtain
the asserted four-term exact sequence.
\end{proof}

The whole-plane obstruction $W$ and the interchange kernel $\ker(\kappa)$
depend on the underlying exact couple and Cartan--Eilenberg system,
respectively, not just the associated spectral sequence.  Boardman gave a
useful sufficient criterion for the vanishing of $W$, which only depends
on data internal to the spectral sequence.  The analogous statement for
$\ker(\kappa)$ admits the following direct proof.

\begin{proposition} [cf.~\cite{Boa99}*{8.1}]
\label{prop:kerkappazero}
Let $(H, \partial)$ be a Cartan--Eilenberg system.  Suppose that $a$
and $b$ are such that $d^r_s \: E^r_s \to E^r_{s-r}$ is zero whenever
$s-r \le a$ and $s > b$.  Then $\ker(\kappa) = 0$.
\end{proposition}

\begin{proof}
Let $w \in \ker(\kappa) \subset \colim_j \lim_i H(i,j)$ be the image
of $w_j = (w_{i,j})_i \in \lim_i H(i,j)$, for a fixed $j \ge a$.
Then $w_{i,j} \mapsto 0$ under $H(i,j) \to \colim_j H(i,j)$, for each
$i \le j$.  Let $w_{i,k}$ denote the image of $w_{i,j}$ under $\eta \:
H(i,j) \to H(i,k)$, for each $k \ge j$.  Choose $\ell \ge b$ so large that
$w_{a,\ell}  = 0$.  Then $w_\ell = (w_{i,\ell})_i \in \lim_i H(i, \ell)$
maps to $w$.  Consider any $i \le a$.  Suppose, for a contradiction, that
$w_{i,\ell} \ne 0$.  Then there is a minimal $s>\ell$ such that $w_{i,s}
= 0$.  By exactness at $H(i,s-1)$ we can choose an $x \in H(s-1,s) = E^1_s$
with $\partial(x) = w_{i,s-1}$.  There is also a minimal $s-r > i$ such
that $w_{s-r,s-1} = 0$.  We know that $s-r \le a$, and
by exactness at $H(i,s-1)$ we can choose a $y
\in H(i,s-r)$ with $\eta(y) = w_{i,s-1}$.
Let $z \in H(s-r-1,s-r) = E^1_{s-r}$ be the image of $y$.
\medskip
$$
\xymatrix{
y \ar@{|->}@(ur,ul)[rrr] \ar@{|->}[d] & w_{i,j} \ar@{|->}[r] \ar@{|->}[dd] & w_{i,\ell} \ar@{|->}[r] \ar@{|->}[dd]
	& w_{i,s-1} \ar@{|->}[r] \ar@{|->}[dd] & 0 \\
z \\
& w_{a,j} \ar@{|->}[r] & 0 \ar@{|->}[r] & 0 \\
& & & & x \ar@{|->}[uuul]!<0pt,4pt>_-{\partial} 
}
$$
By construction there is now a nonzero differential $d^r_s([x]) = [z]$.
This contradicts the hypothesis that these differentials are zero.  Hence
$w_{i,\ell} = 0$ for all $i$, so $w_\ell = 0$ and $w=0$.
Since $w \in \ker(\kappa)$ was arbitrary, this proves that $\ker(\kappa)
= 0$.
\end{proof}

\begin{remark}
This argument may be applied in one internal degree~$t$ at a time:
If there are integers $a(t)$ and $b(t)$ such that
$d^r_s \: (E^r_s)_{t+1} \to (E^r_{s-r})_t$ is zero whenever
$s-r \le a(t)$ and $s > b(t)$, then $\ker(\kappa)_t = 0$.
\end{remark}

\section{Right Cartan--Eilenberg systems}

Let $\bZ_{+\infty} = \bZ \cup \{+\infty\}$, $\bZ_{-\infty} = \bZ \cup \{-\infty\}$ and
$\bZ_{\pm\infty} = \bZ \cup \{\pm\infty\}$, and extend the linear ordering on
$\bZ$ to these sets by letting $+\infty$ and $-\infty$ be the greatest
and least elements, respectively.  Recall Definition~\ref{def:I-system}.

\begin{definition} \label{def:left-right-extended-cess}
A $\bZ_{+\infty}$-system is called a \emph{right Cartan--Eilenberg
system}, a $\bZ_{-\infty}$-system is called a \emph{left 
Cartan--Eilenberg system}, and a $\bZ_{\pm\infty}$-system is called an
\emph{extended Cartan--Eilenberg system}.  By restriction to $\bZ$, each
such system has an \emph{underlying} Cartan--Eilenberg system.
\end{definition}

We discuss right Cartan--Eilenberg systems in this section,
and turn to left Cartan--Eilenberg systems in the next section.

\begin{definition}
Let $(H, \partial)$ be a right Cartan--Eilenberg system.  Let $A''_s
= H(s, \infty)$ and $E^1_s = H(s-1,s)$ for each $s \in \bZ$, and let
$$
\xymatrix{
A''_{s-1} \ar[rr]^-{\alpha_s} && A''_s \ar[dl]^-{\beta_s} \\
& E^1_s \ar[ul]^-{\gamma_s}
}
$$
be given by
\begin{align*}
\alpha_s &= \eta \: H(s-1,\infty) \longto H(s,\infty) \\
\beta_s &= \partial \: H(s, \infty) \longto H(s-1,s) \\
\gamma_s &= \eta \: H(s-1,s) \longto H(s-1,\infty) \,.
\end{align*}
We call $(A'', E^1, \alpha, \beta, \gamma)$ the \emph{right couple}
associated to $(H, \partial)$.
\end{definition}

\begin{lemma}
The right couple $(A'', E^1, \alpha, \beta, \gamma)$ is an exact couple,
and the associated spectral sequence is equal to the one associated to
the underlying Cartan--Eilenberg system of $(H, \partial)$.
\end{lemma}

\begin{proof}
The exact triangles for $(s-1,s,\infty)$ of the right Cartan--Eilenberg
system form the right (exact) couple.  Diagram chases show that
$\gamma^{-1} \im(\alpha^{r-1})
	= Z^r_s = \ker(\partial \: E^1_s \to H(s-r,s-1))$,
that
$\beta \ker(\alpha^{r-1})
	= B^r_s = \im(\partial \: H(s,s+r-1) \to E^1_s)$,
and that the definitions of the $d^r$-differentials agree,
for all $r\ge1$ and $s$.
\end{proof}

\begin{proposition} \label{prop:cc2limit}
Let $(H, \partial)$ be a right (resp.~extended) Cartan--Eilenberg
system. The right couple $(A'', E^1)$ is conditionally convergent to
the limit if and only if
$$
\bar\eta \: \colim_j H(i,j) \to H(i, \infty)
$$
is an isomorphism for some $i \in \bZ$ (resp.~$i \in \bZ_{-\infty}$),
in which case it is an isomorphism for every $i \in \bZ$ (resp.~$i \in
\bZ_{-\infty}$).
\end{proposition}

\begin{proof}
The colimit over~$j$ of the exact triangles for $(i,j,\infty)$
gives an exact triangle
$$
\xymatrix{
\colim_j H(i,j) \ar[rr]^-{\bar\eta}
	&& H(i, \infty) \ar[dl] \\
& \colim_j H(j, \infty) \ar[ul]
}
$$
for each $i$, so $A''_\infty = \colim_j H(j, \infty)$ is zero
if and only if $\bar\eta \: \colim_j H(i,j) \to H(i, \infty)$ is an
isomorphism for some $i$, and this implies that $\bar\eta$ is an
isomorphism for each $i$.
\end{proof}

\begin{lemma}
Each underlying (resp.~left) Cartan--Eilenberg system $(H, \partial)$ can
be prolonged, in an essentially unique way, to a right (resp.~extended)
Cartan--Eilenberg system whose right couple $(A'', E^1)$ is conditionally
convergent to its limit.
\end{lemma}

\begin{proof}
Let
$H(i, \infty) = \colim_j H(i,j)$
for each $i$.  The exact triangle
$$
\xymatrix{
H(i,j) \ar[r]^-{\eta} & H(i, \infty) \ar[d]^-{\eta} \\
& H(j, \infty) \ar[ul]^-{\partial}
}
$$
is the colimit over~$k$ of the exact triangles for $(i,j,k)$.
\end{proof}

\begin{remark}
Cartan and Eilenberg assume in \cite{CE56}*{XV.7} that $H(i,j)$ is
defined for all $-\infty \le i \le j \le \infty$, with $\colim_j H(i,j)
\cong H(i, \infty)$ for all $i \in \bZ_{-\infty}$.  In our terminology
this means that they only consider extended Cartan--Eilenberg systems
with right couples that are conditionally convergent to their limits.
We emphasize the underlying Cartan--Eilenberg systems, with $H(i,j)$
defined for finite~$i$ and~$j$, since this structure suffices to define
the interchange morphism~$\kappa$.
\end{remark}

\begin{theorem} \label{thm:W''iskerkappa}
Let $(H, \partial)$ be a right Cartan--Eilenberg system.  There is
a natural isomorphism
$$
W'' \overset{\cong}\longto \ker(\kappa)
$$
of internal degree~$-1$,
where $W'' = \colim_s \Rlim_r K_\infty \im^r A''_s$ is Boardman's
whole-plane obstruction group for the right couple $(A'', E^1)$, and
$\kappa$ is the interchange morphism.
\end{theorem}

\begin{proof}
For each $i \le j \le k < \infty$ we have a commutative diagram
$$
\xymatrix{
& H(i,j) \ar@{=} [r] \ar[d] & H(i,j) \ar[d] \\
H(k, \infty) \ar[r]^-{\partial} \ar@{=}[d]
	& H(i,k) \ar[r] \ar[d] & H(i,\infty) \ar[r] \ar[d]
	& H(k, \infty) \ar@{=}[d] \\
H(k, \infty) \ar[r]^-{\partial}
	& H(j,k) \ar[r] \ar[d]^-{\partial}
	& H(j,\infty) \ar[r] \ar[d]^-{\partial}
	& H(k, \infty) \\
& H(i,j) \ar@{=}[r] & H(i,j)
}
$$
with exact rows and columns.  Passing to colimits over~$k$ we
get the commutative diagram
$$
\xymatrix{
& H(i,j) \ar@{=} [r] \ar[d]^-{\iota_{i,j}} & H(i,j) \ar[d] \\
A''_\infty \ar[r]^-{\partial} \ar@{=}[d]
        & D_i \ar[r] \ar[d] & A''_i \ar[r] \ar[d]^-{\alpha^{j-i}}
        & A''_\infty \ar@{=}[d] \\
A''_\infty \ar[r]^-{\partial}
        & D_j \ar[r] \ar[d]^-{\partial}
        & A''_j \ar[r]^-{\iota_j} \ar[d]^-{\partial}
        & A''_\infty \\
& H(i,j) \ar@{=}[r] & H(i,j)
}
$$
with exact rows and columns.  Here $D_i = \colim_k H(i,k)$ and $A''_i =
H(i,\infty)$, while $A''_\infty = \colim_k A''_k$.  The homomorphism
$D_j \to A''_j$ maps
$$
\ker(\partial \: D_j \to H(i,j)) = \im(D_i \to D_j) \cong 
\cok(\iota_{i,j}) = C(i,j)
$$
onto
\begin{align*}
\im(D_j \to A''_j) \cap \ker(\partial \: A''_j \to H(i,j)) 
&= \ker(\iota_j \: A''_j \to A''_\infty)
	\cap \im(\alpha^{j-i} \: A''_i \to A''_j) \\
&= K_\infty \im^{j-i} A''_j \,,
\end{align*}
with kernel $\im(\partial \: A''_\infty \to D_j)$.  Hence there is an
exact sequence
$$
A''_\infty \overset{\partial}\longto
	C(i,j) \longto K_\infty \im^{j-i} A''_j \to 0 \,.
$$
By right exactness of $\Rlim_i$, we obtain an exact sequence
$$
\Rlim_i A''_\infty \overset{\partial}\longto
	\Rlim_i C(i,j) \longto \Rlim_i K_\infty \im^{j-i} A''_j \to 0
$$
where $\Rlim_i A''_\infty = 0$.
Hence the right hand map is an isomorphism.  Passing to colimits over~$j$
we obtain the isomorphism
$$
\colim_j \Rlim_i C(i,j)
	\cong \colim_j \Rlim_i K_\infty \im^{j-i} A''_j = W'' \,.
$$
By Theorem~\ref{thm:lambda-kappa-sequence}, the left hand side
is isomorphic to $\ker(\kappa)$.
\end{proof}

\section{Left Cartan--Eilenberg systems}


\begin{definition}
Let $(H, \partial)$ be a left Cartan--Eilenberg system.  Let $A'_s
= H(-\infty, s)$ and $E^1_s = H(s-1,s)$ for each $s \in \bZ$, and let
$$
\xymatrix{
A'_{s-1} \ar[rr]^-{\alpha_s} && A'_s \ar[dl]^-{\beta_s} \\
& E^1_s \ar[ul]^-{\gamma_s}
}
$$
be given by
\begin{align*}
\alpha_s &= \eta \: H(-\infty, s-1) \longto H(-\infty, s) \\
\beta_s &= \eta \: H(-\infty, s) \longto H(s-1,s) \\
\gamma_s &= \partial \: H(s-1,s) \longto H(-\infty,s-1) \,.
\end{align*}
We call $(A', E^1, \alpha, \beta, \gamma)$ the \emph{left couple}
associated to $(H, \partial)$.
\end{definition}

\begin{lemma}
The left couple $(A', E^1, \alpha, \beta, \gamma)$ is an exact couple,
and the associated spectral sequence is equal to the one
associated to the underlying Cartan--Eilenberg system of $(H, \partial)$.
\end{lemma}

\begin{proof}
The exact triangles for $(-\infty,s-1,s)$ of the left Cartan--Eilenberg
system form the left (exact) couple.  The map of exact couples $(A'', E^1)
\to (A', E^1)$ given by $\partial \: A''_s = H(s, \infty) \to H(-\infty,
s) = A'_s$ and $\id \: E^1_s \to E^1_s$, for each $s \in \bZ$, induces
a map of spectral sequences.  It is the identity at the $E^1$-term,
hence is also the identity map at all later $E^r$-terms.
\end{proof}

\begin{proposition} \label{prop:cc2colimit}
Let $(H, \partial)$ be a left (resp.~extended) Cartan--Eilenberg system
whose left couple $(A', E^1)$ is conditionally convergent to its colimit.
Then for each $j \in \bZ$ (resp.~$j \in \bZ_{+\infty}$) there is a short
exact sequence
$$
0 \to \Rlim_i H(i,j) \overset{\partial}\longto H(-\infty,j)
	\overset{\tilde\eta}\longto \lim_i H(i,j) \to 0 \,,
$$
where $\partial$ has internal degree~$-1$.
\end{proposition}

\begin{proof}
Fix~$j$, and consider the exact triangles
$$
\xymatrix{
A'_i \ar[r]^-{\alpha^{j-i}}
	& A'_j \ar[d]^-{\eta} \\
& H(i,j) \ar[ul]^-{\partial}
}
$$
for integers $i$ with $i \le j$.  (Some notational modifications are
appropriate for $j=\infty$, but otherwise the argument is the same.)  Let
\begin{align*}
K'(i,j) &= \ker(\alpha^{j-i} \: A'_i \to A'_j) \,, \\
I'(i,j) &= \im(\alpha^{j-i} \: A'_i \to A'_j) \,, \\
C'(i,j) &= \cok(\alpha^{j-i} \: A'_i \to A'_j) \,,
\end{align*}
leading to the short exact sequences
\begin{align*}
0 &\to K'(i,j) \longto A'_i \longto I'(i,j) \to 0 \,, \\
0 &\to I'(i,j) \longto A'_j \longto C'(i,j) \to 0 \,, \\
0 &\to C'(i,j) \longto H(i,j) \overset{\partial}\longto K'(i,j) \to 0 \,.
\end{align*}
Here $I'(i,j) = \im^{j-i} A'_j$, so $\lim_i I'(i,j) = Q_j$ and $\Rlim_i
I'(i,j) = RQ_j$ in Boardman's notation (Definition~\ref{def:Q-RQ-W}).
Passing to limits over~$i$, we obtain the exact sequences
\begin{align*}
0 &\to \lim_i K'(i,j) \to A'_{-\infty} \to Q_j
	\to \Rlim_i K'(i,j) \to RA'_{-\infty} \to RQ_j \to 0 \,, \\
0 &\to Q_j \to A'_j \overset{\tilde\eta}\longto \lim_i C'(i,j)
	\to RQ_j \to 0 \to \Rlim_i C'(i,j) \to 0 \,, \\
0 &\to \lim_i C'(i,j) \to \lim_i H(i,j) \overset{\partial}\longto
	\lim_i K'(i,j) \to \Rlim_i C'(i,j) \to \Rlim_i H(i,j)
	\overset{\partial}\longto \Rlim_i K'(i,j) \to 0 \,.
\end{align*}
By the assumption of conditional convergence, $A'_{-\infty}
= \lim_j A'_j = 0$ and $RA'_{-\infty} = \Rlim_j A'_j = 0$, so $\lim_i
K'(i,j) = 0$, $Q_j \cong \Rlim_i K'(i,j)$, $RQ_j = 0$, $\Rlim_i C'(i,j) =
0$, $\lim_i C'(i,j) \cong \lim_i H(i,j)$ and $\partial \: \Rlim_i H(i,j)
\cong \Rlim_i K'(i,j)$, so the short exact sequence $0 \to Q_j \to A'_j
\to \lim_i C'(i,j) \to 0$ can be rewritten in the form
$$
0 \to \Rlim_i H(i,j) \overset{\partial}\longto
	A'_j \overset{\tilde\eta}\longto \lim_i H(i,j) \to 0 \,,
$$
as claimed.
\end{proof}

\begin{remark}
There is probably no canonical prolongation of a Cartan--Eilenberg
system $(H, \partial)$ to a left Cartan--Eilenberg system, even if
we require that the left couple $(A', E^1, \alpha, \beta, \gamma)$
should be conditionally convergent to its colimit.  The objects $H(i,j)$
for integers $i\le j$ determine $\lim_i H(i,j)$ and $\Rlim_i H(i,j)$,
but these only determine $A'_j = H(-\infty, j)$ up to an extension.
\end{remark}

\begin{theorem} \label{thm:W'iskerkappa}
Let $(H, \partial)$ be a left Cartan--Eilenberg system.  There
is a natural isomorphism
$$
W' \overset{\cong}\longto \ker(\kappa)
$$
of internal degree~$0$,
where $W' = \colim_s \Rlim_r K_\infty \im^r A'_s$ is
Boardman's whole-plane obstruction group for the
left couple $(A', E^1)$, and $\kappa$ is the interchange morphism.
\end{theorem}

\begin{proof}
The proof is very similar to that of Theorem~\ref{thm:W''iskerkappa}.
For each $-\infty < i \le j \le k$ we have a commutative diagram
$$
\xymatrix{
& H(i,j) \ar@{=}[r] \ar[d] & H(i,j) \ar[d]^-{\partial} \\
H(-\infty,k) \ar[r] \ar@{=}[d] & H(i,k) \ar[r]^-{\partial} \ar[d]
	& H(-\infty,i) \ar[r] \ar[d]
	& H(-\infty,k) \ar@{=}[d] \\
H(-\infty,k) \ar[r]
	& H(j,k) \ar[r]^-{\partial} \ar[d]^-{\partial}
	& H(-\infty,j) \ar[r] \ar[d]
	& H(-\infty,k) \\
& H(i,j) \ar@{=}[r] & H(i,j)
}
$$
with exact rows and columns.
Passing to colimits over $k$ we get the commutative diagram
$$
\xymatrix{
& H(i,j) \ar@{=}[r] \ar[d]^-{\iota_{i,j}} & H(i,j) \ar[d]^-{\partial} \\
A'_\infty \ar[r] \ar@{=}[d] & D_i \ar[r]^-{\partial} \ar[d]
        & A'_i \ar[r] \ar[d]^-{\alpha^{j-i}}
        & A'_\infty \ar@{=}[d] \\
A'_\infty \ar[r]
        & D_j \ar[r]^-{\partial} \ar[d]^-{\partial}
        & A'_j \ar[r]^-{\iota_j} \ar[d]
        & A'_\infty \\
& H(i,j) \ar@{=}[r] & H(i,j)
}
$$
with exact rows and columns.  Here $D_i = \colim_k H(i,k)$, $A'_i =
H(-\infty, i)$ and $A'_\infty = \colim_k A'_k$.  The homomorphism
$\partial \: D_j \to A'_j$ maps
$$
\ker(\partial \: D_j \to H(i,j)) = \im(D_i \to D_j) \cong
\cok(\iota_{i,j}) = C(i,j)
$$
onto
\begin{align*}
\im(\partial \: D_j \to A'_j) \cap \ker(A'_j \to H(i,j))
&= \ker(\iota_j \: A'_j \to A'_\infty)
        \cap \im(\alpha^{j-i} \: A'_i \to A'_j) \\
&= K_\infty \im^{j-i} A'_j \,,
\end{align*}
with kernel $\im(A'_\infty \to D_j)$.  Hence there is an
exact sequence
$$
A'_\infty \longto C(i,j) \overset{\partial}\longto
	K_\infty \im^{j-i} A'_j \to 0 \,.
$$
By right exactness of $\Rlim_i$, we obtain an exact sequence
$$
\Rlim_i A'_\infty \longto \Rlim_i C(i,j)
	\overset{\partial}\longto \Rlim_i K_\infty \im^{j-i} A'_j \to 0 \,,
$$
where $\Rlim_i A'_\infty = 0$.
Hence the right hand map is an isomorphism.  Passing to colimits over~$j$
we obtain the isomorphism
$$
\partial \: \colim_j \Rlim_i C(i,j) \overset{\cong}\longto
	\colim_j \Rlim_i K_\infty \im^{j-i} A'_j = W' \,.
$$
By Theorem~\ref{thm:lambda-kappa-sequence}, the left hand side
is isomorphic to $\ker(\kappa)$.
\end{proof}

\begin{definition}
An extended Cartan--Eilenberg system $(H, \partial)$ is
\emph{conditionally convergent} if the left couple $(A', E^1)$ is
conditionally convergent to the colimit and the right couple $(A'',
E^1)$ is conditionally convergent to the limit.
\end{definition}

\begin{remark}
We summarize the discussion.
The spectral sequence $(E^r, d^r)_r$ associated to a extended
Cartan--Eilenberg system $(H, \partial)$ has three plausible target
groups, namely $A'_\infty = \colim_j H(-\infty, j)$, $H(-\infty, \infty)$
and $A''_{-\infty} = \lim_i H(i, \infty)$.
$$
\xymatrix{
A'_i \ar[r]^-{\alpha^{j-i}}
	& A'_j \ar[r]^-{\iota_j} \ar[d]^-{\tilde\eta}
	& A'_\infty \ar[rr]^-{\bar\eta} \ar[d]
	&& H(-\infty,\infty) \ar[dd]^-{\tilde\eta} \\
& \lim_i H(i, j) \ar[r]^-{\iota_j} \ar[dd]^-{\pi_{i,j}}
	& \colim_j \lim_i H(i, j) \ar[dr]^-{\kappa} \\
& & & \lim_i \colim_j H(i,j) \ar[r] \ar[d]^-{\pi_i}
	& A''_{-\infty} \ar[d]^-{\pi_i} \\
& H(i,j) \ar[rr]^-{\iota_{i,j}}
	&& \colim_j H(i,j) \ar[r]^-{\bar\eta}
	& A''_i \ar[d]^-{\alpha^{j-i}} \\
& && & A''_j
}
$$
When $(H, \partial)$ is conditionally convergent,
Theorems~\ref{thm:GH2.2.5} and~\ref{thm:GH2.2.4} apply to the left
couple $(A', E^1)$ and to the right couple $(A'', E^1)$, respectively.
Hence there are isomorphisms $\colim_j H(i,j) \cong H(i,\infty)$ for all
$i \in \bZ_{-\infty}$, and short exact sequences $0 \to \Rlim_i H(i,j)
\to H(-\infty, j) \to \lim_i H(i,j) \to 0$ for all $j \in \bZ_{+\infty}$.
In particular,
$$
\bar\eta \: A'_\infty \overset{\cong}\longto H(-\infty, \infty)
$$
in an isomorphism and
$$
0 \to RA''_{-\infty} \overset{\partial}\longto
	H(-\infty, \infty) \overset{\tilde\eta}\longto A''_{-\infty} \to 0
$$
is exact.
Under the additional hypothesis $RE^\infty = 0$, we know that $(E^r,
d^r)_r$ converges weakly to a complete filtration of $A'_\infty$,
with $F_{-\infty} A'_\infty \cong W$.  We also know that $(E^r, d^r)$
converges strongly to $A''_{-\infty}$, with $RA''_{-\infty} \cong W$.
Hence, if $W = 0$ then $(E^r, d^r)$ converges strongly to $A'_\infty
\cong H(-\infty,\infty) \cong A''_{-\infty}$.  When $W \ne 0$, the
spectral sequence is not strongly convergent to $A'_\infty$, because
the filtration $\{F_s A'_\infty\}_s$ fails to be Hausdorff, with limit
$F_{-\infty} A'_\infty \cong W$.  The spectral sequence is strongly
convergent to $A''_{-\infty}$, which is the quotient of $H(-\infty,
\infty)$ by $RA''_{-\infty} \cong W$.  Hence $H(-\infty,\infty)$ plays
the role of the ideal target group, and $W \cong \ker(\kappa)$ precisely
measures the failure of the spectral sequence to converge strongly to
this target.
\end{remark}

\section{Sequences of spectra}

Cartan--Eilenberg systems naturally arise from filtered objects, and
are well suited for the construction of multiplicative spectral sequences.
Consider a biinfinite sequence of spectra
$$
X_{-\infty} \to \dots \to X_{s-1} \to X_s \to \dots \to X_\infty \,,
$$
with $X_{-\infty} = \holim_s X_s$ and $X_\infty = \hocolim_s X_s$.
We obtain an extended Cartan--Eilenberg system $(H, \partial)$, with
$$ H(i,j) = \pi_*(\cone(X_i \to X_j)) $$ for $-\infty \le i \le j
\le \infty$, and two exact couples $(A', E^1)$ and $(A'', E^1)$, with
\begin{align*} A'_s &= \pi_*(\cone(X_{-\infty} \to X_s)) \\ A''_s &=
\pi_*(\cone(X_s \to X_\infty)) \end{align*} for $s \in \bZ$.  The three
associated spectral sequences are all equal, and begin with
$$
E^1_s = \pi_*(\cone(X_{s-1} \to X_s)) \,.
$$
A typical aim is to calculate $G = \pi_*(X_\infty)$ under the assumption
that $X_{-\infty} \simeq *$.  Under these hypotheses the extended
Cartan--Eilenberg system $(H, \partial)$ is conditionally convergent,
because of the short exact sequence
$$
0 \to \Rlim_s \pi_{*+1}(X_s) \longto \pi_*(X_{-\infty})
	\longto \lim_s \pi_*(X_s) \to 0
$$
and the isomorphism
$$
\colim_s \pi_*(X_s) \overset{\cong}\longto \pi_*(X_\infty) \,.
$$

The first (left) exact couple $(A', E^1)$ is conditionally
convergent to the colimit
$$
A'_\infty = \colim_s \pi_*(\cone(X_{-\infty} \to X_s)) \cong G \,,
$$
which is the target of interest, cf.~Theorem~\ref{thm:GH2.2.5}, which is
a variant of \cite{Boa99}*{8.10}.  When $RE^\infty = 0$, which can often
be verified from the differential structure in the spectral sequence,
the spectral sequence is weakly convergent to a complete filtration $\{F_s
A'_\infty\}_s$ of $G$, but the filtration may fail to be Hausdorff.  We can
therefore only hope to recover the quotient $G/F_{-\infty} A'_\infty$,
where $F_{-\infty} A'_\infty = \lim_s F_s A'_\infty$ is the limit of the
filtration.  In this case there is an isomorphism $W' \cong F_{-\infty}
A'_\infty$, where $W'$ is Boardman's group for the exact couple $(A',
E^1)$, and our Theorem~\ref{thm:W'iskerkappa} identifies this error term
with $\ker(\kappa)$.

The second (right) exact couple $(A'', E^1)$ is conditionally convergent to
the limit
$$
A''_{-\infty} = \lim_s \pi_*(\cone(X_s \to X_\infty)) \cong G/RA''_{-\infty}
\,,
$$
where $RA''_{-\infty} = \Rlim_s \pi_{*+1}(\cone(X_s \to X_\infty))$,
cf.~Theorem~\ref{thm:GH2.2.4}, which is a variant of \cite{Boa99}*{8.13}.
When $RE^\infty = 0$ the spectral sequence is strongly convergent to this
limit.  Since $G$ is the group we are principally interested in, we also
need to understand the subgroup $RA''_{-\infty}$.  In this case there is
an isomorphism $W'' \cong RA''_{-\infty}$, where $W''$ is Boardman's group
for the exact couple $(A'', E^1)$, and our Theorem~\ref{thm:W''iskerkappa}
identifies this error term with $\ker(\kappa)$.

\begin{example}
Let $Hk$ be the Eilenberg--Mac\,Lane spectrum of a field~$k$,
and let
$$
X_s = \prod_{i \ge |s|} Hk
$$
for $s, i \in \bZ$.  Let $X_{s-1} \to X_s$ be given by the identity map
on the $i$-th factor, except when $s \le 0$ and $i = |s|$, when it is
given by $* \to Hk$, and when $s > 0$ and $i = |s|-1$, when it is given
by $Hk \to *$.  Then
$$
X_\infty \simeq \cone(\bigvee_{i\ge0} Hk \to \prod_{i\ge0} Hk)
$$
and $X_{-\infty} \simeq *$.  We obtain an exact couple $(A', E^1)$
converging conditionally to colimit, with $A'_s = \prod_{i\ge |s|} k$,
$A'_\infty = \prod_{i\ge0} k/\bigoplus_{i\ge0} k$, $A'_{-\infty} = 0$
and $RA'_{-\infty} = 0$.  Here $E^1_s = k$ for $s\le 0$ and $E^1_s =
\Sigma k$ for $s\ge1$, with $d^{2r-1}_r \: E^{2r-1}_r \to E^{2r-1}_{1-r}$
an isomorphism for each $r\ge1$.  Hence $E^\infty = 0$ and $RE^\infty
= 0$.  Also $\iota_s \: A'_s \to A'_\infty$ is surjective for each $s$,
so $F_s A'_\infty = A'_\infty$.  Hence
$$
W' = F_{-\infty} A'_\infty = F_\infty A'_\infty
= \prod_{i\ge0} k/\bigoplus_{i\ge0} k
$$
and $RF_{-\infty} A'_\infty = 0$.  Note that this is not a half-plane
spectral sequence with exiting or entering differentials, in the sense
of \cite{Boa99}*{\S6, \S7}, even if the $E^1$-term is concentrated
in the intersection of the two half-planes $t\ge0$ and $t\le1$, where $t$
is the internal degree.
\end{example}

\begin{example}
Let $J_p = L_{K(1)} S$ be the image-of-$J$ spectrum, completed at an
odd prime~$p$.  Its homotopy groups are
$$
\pi_n(J_p) \cong \begin{cases}
\bZ_p & \text{for $n \in \{-1,0\}$,} \\
\bZ/p^{m+1} & \text{for $n = (2p-2)p^m q - 1$, $p \nmid q$,} \\
0 & \text{otherwise,}
\end{cases}
$$
and its mod~$p$ homotopy groups form the graded ring
$$
\pi_*(J/p) = \bZ/p[\alpha, v^{\pm1}]/(\alpha^2) \,,
$$
with $|\alpha| = 2p-3$ and $|v| = 2p-2$.  Let $J_p^{tS^1} = [\tilde ES^1
\wedge F(ES^1_+, J_p)]^{S^1}$ denote the Tate construction \cite{GM95} for
the trivial $S^1$-action on $J_p$.  The biinfinite Greenlees filtration of
$\tilde ES^1$ leads to a sequence of spectra with associated whole-plane
Tate spectral sequence
$$
\hat E^2_{s,n}(S^1, J_p)
	\Longrightarrow_s \pi_{s+n}(J_p^{tS^1}) \,.
$$
Any filtration of $\bZ_p$ is complete, so $RE^\infty = 0$.  For bidegree
reasons the only nonzero differentials are of the form $d^r_{s,n}$ with
$s\ne0$ even and $n=0$, so $W = 0$ by Boardman's criterion.  Hence this
spectral sequence is strongly convergent, to the abutment calculated by
Hesselholt and Madsen in~\cite{HM92}*{0.2}.  On the other hand,
the $S^1$-Tate spectral sequence for $J/p$ is
\begin{align*}
\hat E^2_{s,n}(S^1, J/p)
	&= \bZ/p[t^{\pm1}] \otimes
        \bZ/p[\alpha, v^{\pm1}]/(\alpha^2) \\
        &\Longrightarrow_s \pi_{s+n}(J^{tS^1}/p) \,.
\end{align*}
B{\"o}kstedt and Madsen \cite{BM94} showed that
it has nonzero differentials
$$
d^{2(p^{k+1}-1)}(t^{p^k-p^{k+1}}) \doteq
	v^{p(p^k-1)/(p-1)} \cdot t^{p^k-1} \alpha
$$
(up to units in $\bZ/p$) for each $k\ge0$, and
the classes $\alpha$ and $v$ are infinite cycles.
Hence
$$
\hat E^\infty_{*,*}(S^1, J/p)
	= \bZ/p[v^{\pm1}] \{1, t^{-1} \alpha\}
$$
and $RE^\infty = 0$.  The spectral sequence is thus weakly convergent to
$G = \bZ/p[v^{\pm1}] \{1, t^{-1} \alpha\}$, for a complete filtration that
might not be Hausdorff.  Boardman's criterion for $W = 0$ applies for the
differentials landing in even total degrees, but not for the differentials
landing in odd total degrees.  Indeed, the surjection $\pi_*(J^{tS^1}/p)
\onto G$ with kernel $W \cong F_{-\infty} A_\infty$ is an isomorphism
in even degrees, but has the large kernel $(\prod_{i\ge0}
\bZ/p)/(\bigoplus_{i\ge0} \bZ/p)$ in each odd degree \cite{HM92}*{4.4}.
\end{example}

\begin{example}
Similar patterns are found in the Tate spectral sequences for $\hat
T(\bZ)^{tS^1}$ and $\hat T(\bZ)^{tS^1}/p$ for odd primes~$p$, where
$T(\bZ) = THH(\bZ)$ denotes the topological Hochschild homology
spectrum of the integers and $\varphi_p \: T(\bZ) \to \hat T(\bZ)
= T(\bZ)^{tC_p}$ is its $p$-cyclotomic structure map, in the sense
of \cite{NS}*{II.1.1}.  Here $TP(\bZ) = T(\bZ)^{tS^1}$ is expected
to have deep arithmetic significance,
by analogy with the results for smooth and proper schemes over finite
fields in~\cite{Hes}, but the structure of its Tate spectral sequence
is not fully known, cf.~\cite{Rog98}*{\S3} and~\cite{Rog99}*{\S1}.
The $S^1$-Tate spectral sequence
$$
\hat E^2_{*,*}(S^1, \hat T(\bZ))
	\Longrightarrow \pi_*(\hat T(\bZ)^{tS^1})
$$
is strongly convergent.  The mod~$p$ spectral sequence
\begin{align*}
\hat E^2_{*,*}(S^1, \hat T(\bZ)/p)
	&= \bZ/p[t^{\pm1}] \otimes \bZ/p[e,f^{\pm1}]/(e^2) \\
	&\Longrightarrow \pi_*(\hat T(\bZ)^{tS^1}/p) \,,
\end{align*}
where $|e| = 2p-1$ and $|f| = 2p$, is only known to be weakly convergent
for a complete filtration.  B{\"o}kstedt and Madsen \cite{BM95} showed
that the latter spectral sequence has nonzero differentials
$$
d^{2p(p^{k+1}-1)/(p-1)} (t^{p^k-p^{k+1}})
	\doteq (tf)^{p(p^k-1)/(p-1)} \cdot t^{p^k}e
$$
(up to units in $\bZ/p$), for each $k\ge0$.  Hence
$$
\hat E^\infty_{*,*}(S^1, \hat T(\bZ)/p)
	= \bZ/p[(tf)^{\pm1}] \{1, e\} \,,
$$
and there is a surjection $\pi_*(\hat T(\bZ)^{tS^1}/p) \onto G =
\bZ/p[(tf)^{\pm1}] \{1, e\}$, with kernel $W \cong F_{-\infty} A_\infty$.
By Boardman's criterion, $W$ is again zero in even total degrees, but may,
very well, be nonzero in odd total degrees.
\end{example}


\begin{bibdiv}
\begin{biblist}

\bib{Boa99}{article}{
   author={Boardman, J. Michael},
   title={Conditionally convergent spectral sequences},
   conference={
      title={Homotopy invariant algebraic structures},
      address={Baltimore, MD},
      date={1998},
   },
   book={
      series={Contemp. Math.},
      volume={239},
      publisher={Amer. Math. Soc., Providence, RI},
   },
   date={1999},
   pages={49--84},
   review={\MR{1718076}},
}

\bib{BM94}{article}{
   author={B\"okstedt, M.},
   author={Madsen, I.},
   title={Topological cyclic homology of the integers},
   note={$K$-theory (Strasbourg, 1992)},
   journal={Ast\'erisque},
   number={226},
   date={1994},
   pages={7--8, 57--143},
   issn={0303-1179},
   review={\MR{1317117}},
}

\bib{BM95}{article}{
   author={B\"okstedt, M.},
   author={Madsen, I.},
   title={Algebraic $K$-theory of local number fields: the unramified case},
   conference={
      title={Prospects in topology},
      address={Princeton, NJ},
      date={1994},
   },
   book={
      series={Ann. of Math. Stud.},
      volume={138},
      publisher={Princeton Univ. Press, Princeton, NJ},
   },
   date={1995},
   pages={28--57},
   review={\MR{1368652}},
}

\bib{CE56}{book}{
   author={Cartan, Henri},
   author={Eilenberg, Samuel},
   title={Homological algebra},
   publisher={Princeton University Press, Princeton, N. J.},
   date={1956},
   pages={xv+390},
   review={\MR{0077480}},
}

\bib{GM95}{article}{
   author={Greenlees, J. P. C.},
   author={May, J. P.},
   title={Generalized Tate cohomology},
   journal={Mem. Amer. Math. Soc.},
   volume={113},
   date={1995},
   number={543},
   pages={viii+178},
   issn={0065-9266},
   review={\MR{1230773}},
}

\bib{HM92}{article}{
   author={Hesselholt, Lars},
   author={Madsen, Ib},
   title={The $S^1$-Tate spectrum for $J$},
   note={Papers in honor of Jos\'e Adem (Spanish)},
   journal={Bol. Soc. Mat. Mexicana (2)},
   volume={37},
   date={1992},
   number={1-2},
   pages={215--240},
   review={\MR{1317575}},
}

\bib{Hes}{article}{
   author={Hesselholt, Lars},
   title={Topological Hochschild homology and the Hasse-Weil zeta function},
   note={{\tt arXiv:1602.01980}, to appear},
}

\bib{Kos47}{article}{
   author={Koszul, Jean-Louis},
   title={Sur les op\'erateurs de d\'erivation dans un anneau},
   language={French},
   journal={C. R. Acad. Sci. Paris},
   volume={225},
   date={1947},
   pages={217--219},
   review={\MR{0022345}},
}

\bib{Ler46}{article}{
   author={Leray, Jean},
   title={Structure de l'anneau d'homologie d'une repr\'esentation},
   language={French},
   journal={C. R. Acad. Sci. Paris},
   volume={222},
   date={1946},
   pages={1419--1422},
   review={\MR{0016665}},
}

\bib{LurHA}{book}{
   author={Lurie, Jacob},
   title={Higher Algebra},
   note={{\tt www.math.harvard.edu/$\sim$lurie/papers/HA.pdf}},
}

\bib{ML98}{book}{
   author={Mac Lane, Saunders},
   title={Categories for the working mathematician},
   series={Graduate Texts in Mathematics},
   volume={5},
   edition={2},
   publisher={Springer-Verlag, New York},
   date={1998},
   pages={xii+314},
   isbn={0-387-98403-8},
   review={\MR{1712872}},
}

\bib{Mas52}{article}{
   author={Massey, W. S.},
   title={Exact couples in algebraic topology. I, II},
   journal={Ann. of Math. (2)},
   volume={56},
   date={1952},
   pages={363--396},
   issn={0003-486X},
   review={\MR{0052770}},
}

\bib{Nee02}{article}{
   author={Neeman, Amnon},
   title={A counterexample to a 1961 ``theorem'' in homological algebra},
   note={With an appendix by P. Deligne},
   journal={Invent. Math.},
   volume={148},
   date={2002},
   number={2},
   pages={397--420},
   issn={0020-9910},
   review={\MR{1906154}},
}

\bib{NS}{article}{
   author={Nikolaus, Thomas},
   author={Scholze, Peter},
   title={On topological cyclic homology},
   note={{\tt arXiv:1707.01799}},
}

\bib{Rog98}{article}{
   author={Rognes, John},
   title={Trace maps from the algebraic $K$-theory of the integers (after
   Marcel B\"okstedt)},
   journal={J. Pure Appl. Algebra},
   volume={125},
   date={1998},
   number={1-3},
   pages={277--286},
   issn={0022-4049},
   review={\MR{1600028}},
}

\bib{Rog99}{article}{
   author={Rognes, John},
   title={Topological cyclic homology of the integers at two},
   journal={J. Pure Appl. Algebra},
   volume={134},
   date={1999},
   number={3},
   pages={219--286},
   issn={0022-4049},
   review={\MR{1663390}},
}

\bib{Roo61}{article}{
   author={Roos, Jan-Erik},
   title={Sur les foncteurs d\'eriv\'es de $\underleftarrow\lim$. Applications},
   language={French},
   journal={C. R. Acad. Sci. Paris},
   volume={252},
   date={1961},
   pages={3702--3704},
   review={\MR{0132091}},
}

\bib{Roo06}{article}{
   author={Roos, Jan-Erik},
   title={Derived functors of inverse limits revisited},
   journal={J. London Math. Soc. (2)},
   volume={73},
   date={2006},
   number={1},
   pages={65--83},
   issn={0024-6107},
   review={\MR{2197371}},
}

\end{biblist}
\end{bibdiv}
\end{document}